%% file: quintic.tex
\documentclass[english]{ourlematema}

\usepackage{amsmath, amsthm, amssymb, hyperref, color}
\usepackage{amsfonts,mathrsfs}
 \usepackage{amsthm}
 \usepackage{amssymb}
 \usepackage{enumerate}
 \usepackage{enumitem}
 \usepackage{tikz-cd}
 \usepackage{cleveref}
 \usepackage{float}
 \usepackage{graphicx,color}
 \usepackage{fancyvrb}
 
\theoremstyle{definition}
\newtheorem{Def}{Definition}[section]

\newtheorem{algorithm}[Def]{Algorithm}
\theoremstyle{plain}

\input{operators}
\input{fonts}

\title{Nodes on quintic spectrahedra}
\titlemark{Nodes on quintic spectrahedra}



\author{Taylor Brysiewicz}
\address{Max Planck Institute for Mathematics in the Sciences, Leipzig, Germany\\ \email{taylor.brysiewicz@mis.mpg.de}}

\author{Khazhgali Kozhasov}
\address{Technische Universit\"at Braunschweig, Germany\\ \email{k.kozhasov@tu-braunschweig.de}}

\author{Mario Kummer}
\address{Technische Universit\"at Dresden, Germany\\ \email{mario.kummer@tu-dresden.de}}

\date{2020/09/08}

\begin{document}

\maketitle

\begin{abstract}%
  \noindent
We classify transversal quintic spectrahedra by the location of $20$ nodes on the respective real determinantal surface of degree $5$.  We identify $65$ classes of such surfaces and find an explicit representative in each of them.
\end{abstract}

\section{Introduction}

Spectrahedra are \emph{nonlinear} generalizations of convex polytopes.
They are convex sets that are obtained as intersections of the \emph{cone $ \mathbb{S}^n_{\succcurlyeq}\subset \mathbb{S}^n $ of positive semidefinite matrices} with a real affine linear subspace $\mathcal{L}\subset \mathbb{S}^n$ of symmetric matrices.
Spectrahedra naturally appear in optimization: they are feasibility regions for \emph{semidefinite programs} in the same way as polyhedra are feasibility regions for \emph{linear programs}.
Apart from optimization, spectrahedra have appearances in algebraic geometry, convex geometry, statistics, and combinatorics.
It is conjectured that the class of spectrahedra is so rich that it contains \emph{hyperbolicity cones} of hyperbolic polynomials \cite{HV2007}.
In general, it is of interest to study spectrahedra in any dimension, however, even spectrahedra in dimension $3$ are far from being completely understood.
In the present work we are mainly interested in $3$-dimensional spectrahedra and work in the setting of projective geometry. In particular, we assume that $\mathcal{L}\subset \mathbb{S}^n$ is a real linear subspace of dimension $4$, whose corresponding spectrahedron is $\pp(\mathcal{L}\cap \mathbb{S}^n_{\succcurlyeq})$.

Spectrahedra are intimately related to algebraic surfaces called \emph{symmetroids}.
A \emph{symmetroid} is a surface in projective $3$-space whose defining polynomial is the determinant of a symmetric matrix of linear forms
\begin{align}\label{eq:LSSM}
  A(t,\bar{x})\ =\ t A_0 + x_1 A_1 + x_2 A_2 + x_3 A_3,\quad A_0, A_1, A_2, A_3\in \mathbb{S}^n.
\end{align}
The spectrahedron associated with \eqref{eq:LSSM} is the set $\{(t,\bar{x}) \,:\, A(t,\bar{x}) \succeq 0\}$. If this has nonempty interior, then 
$\det(A(t,\bar{x}))$ is a \emph{spectrahedral symmetroid} containing the algebraic boundary of the spectrahedron.

A generic symmetroid of degree $n$ carries ${n+1 \choose 3}$ \emph{nodal singularities} (points of multiplicity $2$) which are represented by matrices $A(t,\bar{x})$ of corank $2$.
Any symmetroid with this property will be called \emph{transversal}.
If a transversal symmetroid is spectrahedral, then $\rho\leq {n+1\choose 3}$ many of its singular points are real, amongst which $\sigma\leq \rho$ lie on the Euclidean boundary of the associated spectrahedron.
Characterizing possible \emph{combinatorial types} $(\rho, \sigma)$ of spectrahedral symmetroids of a given degree $n$ is a natural problem of real algebraic geometry. Combinatorial types of generic quartic ($n=4$) spectrahedral symmetroids were classified in \cite{DI2011}, see also \cite{quartic} for an alternative proof and explicit examples of symmetroids realizing each combinatorial type.

The main result of the present work is to provide a classification of combinatorial types of transversal quintic ($n=5$) spectrahedral symmetroids.

\begin{thm}\label{thm:class}
There exists a transversal quintic spectrahedral symmetroid of type $(\rho,\sigma)$ if and only if $0\leq \sigma \leq \rho\leq 20$, both $\rho$ and $\sigma$ are even, and $2\leq \rho$.
\end{thm}

In Section \ref{section:prel} we determine restrictions that a combinatorial type of a spectrahedral symmetroid of any given degree $n$ has to satisfy (see Corollary \ref{cor:conditions}). For $n=5$ these are conditions given in Theorem \ref{thm:class}. In Section \ref{section:computations}, using a \emph{hill-climbing algorithm} we find quintic symmetroids realizing each of the $65$ possible types explicitly and certify the findings \textit{a posteriori}. The code and files for reproducing the certification may be found at \href{https://mathrepo.mis.mpg.de/QuinticSpectrahedra/index.html}{https://mathrepo.mis.mpg.de/}.

In Section \ref{section:symmetry} we study real quintic symmetroids with \emph{tetrahedral} and \emph{3-prismatic} symmetries.
In both cases singular points split up into orbits of the action of the corresponding subgroup of the symmetric group.
In particular, if an orbit is real, it either entirely lies on the topological boundary of the associated spectrahedron or is disjoint from it. 
Transversal symmetroids with tetrahedral symmetry form a one-parameter family and those with $3$-prismatic symmetry form a two-parameter family.
In each case we identify all possible configurations of orbits of singular points in terms of the parameters.
\section{Preliminary results}\label{section:prel}
We consider the vector space $\Sn$ of $n\times n$ complex symmetric matrices. We denote by $W_r\subset\pp(\Sn)$ the projective variety of matrices of rank at most $r$. Furthermore, by $\SnR\subset \Sn$ we denote the (real) subspace of real symmetric matrices and $\mathbb{S}^n_{\succ 0}$, $\mathbb{S}^n_{\succcurlyeq 0}\subset \SnR$ stand for the cones of positive definite and positive semidefinite matrices respectively. Before we focus on the case $n=5$, we state some general results. Let us first recall the following result from \cite{highcodim}.

\begin{lemma}
 The codimension of $W_r$ in $\pp(\Sn)$ equals $\binom{n-r+1}{2}$.
\end{lemma}

\begin{Def}
  Given a real point $p=[A]\in\pp(\mathbb{S}^{\,n})$, $A\in \SnR$, the \emph{signature} of $p$ is the unordered pair $\{n_{+},n_{-}\}$ where $n_+$ and $n_-$ are the numbers of positive resp. negative eigenvalues of $A$. We call $p$ \emph{spectrahedral} if either $n_+$ or $n_-$ is zero and \emph{strictly spectrahedral} if it has signature  $\{n,0\}$. For a (strictly) spectrahedral $p=[A]$ one can take $A\in \mathbb{S}^n_{\succcurlyeq 0}$ (resp. $A\in \mathbb{S}^n_{\succ 0}$).
\end{Def}

In the following, we assume $n \geq 3$. 
Let $L\subset\pp(\Sn)$ be a sufficiently generic \emph{real} subspace of dimension $4$, then $C=L\cap W_{n-2}$ is a curve. 

\begin{lemma}
 The signature is constant on each connected component of $C(\R)$. 
\end{lemma}

\begin{proof}
 Since $\codim(W_{n-3})=6$, every $p\in C$ has corank exactly $2$. Therefore, because the eigenvalues of a symmetric matrix depend continuously on their entries, we can find an open neighbourhood of each $p\in C(\R)$ where the signature is constant. Thus the signature is constant on each connected component.
\end{proof}

\begin{lemma}\label{lem:curve}
 Every connected component of $C(\R)$ which is spectrahedral either realizes the trivial homology class in $H_1(\pp(\SnR) ; \Z_2)$ or is an isolated point.
\end{lemma}

\begin{proof}
 The (real) hyperplane $H\subset\pp(\Sn)$ of all traceless matrices contains no spectrahedral points. In particular, it is disjoint from any spectrahedral connected component of $C(\R)$. This implies the claim.
\end{proof}

\begin{cor}\label{cor:bound}
 Let $V\subset\pp(\Sn)$ be a generic real linear subspace of dimension $3$. Then $V\cap W_{n-2}$ consists of $\binom{n+1}{3}$ (complex) points. The number of real points in $V\cap W_{n-2}$ that are spectrahedral is even. Furthermore, the number of real points in $V\cap W_{n-2}$ is of the same parity as ${n+1\choose 3}$. 
\end{cor}

\begin{proof}
  The degree of $W_{n-2}$ was computed in \cite[Prop.~12(b)]{highcodim}. By genericity we can write $V\cap W_{n-2}$ as the intersection of $C$ with a hyperplane that does not contain any isolated real points of $C$. By Lemma \ref{lem:curve} every spectrahedral connected component of $C(\R)$  intersects a generic hyperplane in an even number of points.

  The last assertion follows from the fact that $V$ is real and nonreal points in $V\cap W_{n-2}$ come in  complex conjugate pairs.
\end{proof}

\begin{lemma}\label{smoothspec}
 For each $n\geq4$ there is a full-dimensional set of real $3$-dimensional linear subspaces $V\subset\pp(\Sn)$ which contain a strictly spectrahedral point such that $V\cap W_{n-2}$ contains no spectrahedral points. 
\end{lemma}

\begin{proof}
 We proceed by induction on $n$. For $n=4$ this was shown in \cite{quartic}. Assume that $V\subset\pp(\Sn)$ is such a real $3$-dimensional linear subspace. 
 For $A\in\Sn$ let $$A'=\begin{pmatrix}A & 0 \\ 0 & \tr(A)
\end{pmatrix}.$$ We claim that the subspace $V'\subset\pp(\mathbb{S}^{n+1})$ of all $[A']$ with $[A]\in V$
has the desired properties. Clearly $[A]$ is (strictly) spectrahedral if and only if $[A']$ is (strictly) spectrahedral. Thus $V'$ contains a strictly spectrahedral point because $V$ contains such a point. Now assume that there is a matrix $0\neq A\in\Sn$ such that $[A']\in V'\cap W_{n-1}$ is spectrahedral. This implies that $A$ is positive (or negative) semidefinite and thus $\tr(A)\neq0$. Because $A'$ has corank at least $2$, the matrix $A$ has corank at least $2$ contradicting the assumption that $V\cap W_{n-2}$ contains no spectrahedral points. Any small perturbation of $V'$ also has the desired properties.
\end{proof}

Whereas for $n\geq 4$ there always exist generic real $3$-dimensional linear subspaces $V\subset\pp(\Sn)$ such that $V\cap W_{n-2}$ contains no spectrahedral points, only for some special values of $n$ the real locus of $V\cap W_{n-2}$ can be empty.
\begin{lemma}
\label{lemma:rhopositive}
There exists a real $3$-dimensional linear subspace $V\subset \pp(\Sn)$ containing a strictly spectrahedral point and such that the real locus of $V\cap W_{n-2}$ is empty if and only if $n=-1, 0, 1\ \mathrm{mod}\ 8$.
\end{lemma}
\begin{proof}
  Let $p=[C^\mathsf{T}C]\in V$ be a strictly spectrahedral point in $V$, where $C$ denotes a nonsingular real $n\times n$ matrix. Since congruence by a non-singular matrix leaves the varieties $W_r$ invariant, we have for the \emph{congruent space} $$V_C = \left\{[(C^\mathsf{T})^{-1}AC^{-1}]\,:\, [A]\in V\right\}$$ that $V_C\cap W_{n-2} = (V\cap W_{n-2})_C = \{[(C^\mathsf{T})^{-1}AC^{-1}]\,:\,[A]\in V\cap W_{n-2}\}$. We can assume that $C=\mathrm{Id}$ is the identity matrix, otherwise we replace $V$ with $V_C$ and $p=[C^\mathsf{T}C]\in V$ with $[\mathrm{Id}]\in V_C$. Let $\mathrm{Id}, A_1, A_2, A_3\in \SnR$ be a basis of the linear subspace $V\subset \pp(\Sn)$. Observe that for $(t,\bar{x})=(t,x_1,x_2,x_3)\in \R^4$, a real matrix $$A(t,\bar{x}) = t\, \mathrm{Id}+x_1 A_1+x_2 A_2+ x_3 A_3$$ is of corank at least two (that is, $[A(t,\bar x)]\in V\cap W_{n-2}$) if and only if the matrix $$A(0,\bar x)=x_1A_1+x_2A_2+x_3A_3$$ has a repeated (multiple) eigenvalue. Therefore, there exists a real linear subspace $V\subset \pp(\Sn)$ of dimension $3$ ($V=\left\{[A(t,\bar x)]\,:\, (t,\bar x)=(x_0,x_1,x_2,x_3)\right\}$ without loss of generality) as in the statement if and only there is a real $2$-dimensional plane $v\subset \pp(\Sn)$ ($v=\left\{[A(0,\bar x)]\,:\,\bar{x}=(x_1,x_2,x_3)\right\}$) containing only real matrices with simple (pairwise distinct) eigenvalues. By the main result from \cite{FRS84} existence of such a $2$-plane in $\pp(\Sn)$ is equivalent to $\sigma(n)>2$, where 
  \begin{equation}
  \sigma(n)=\begin{cases}
    2,\ &n\neq 0, \pm 1\ (\mathrm{\normalfont{mod}}\ 8),\\
    \rho(4b),\ &n=8b, 8b\pm 1.
  \end{cases}
\end{equation}
Here $\rho(m)$ is the \emph{Radon-Hurwitz number} of the natural number $m$, that is,
\begin{equation}\label{eq:RH}
  \begin{aligned}
    \rho(m)=2^c+8d,\quad m=(2a+1)2^{c+4d}
  \end{aligned},
\end{equation}
where $a, c,$ and $d$ are integers and $c\in \{0,1,2,3\}$. Finally, for $n=8b, 8b\pm 1$ one always has $\sigma(n)=\rho(4b)>2$, which is straightforward to check using  \eqref{eq:RH}.
\end{proof}
In general, the real part of the projective hypersurface defined by a \emph{hyperbolic polynomial} of degree $n$ can have at most $\lfloor \frac{n+1}{2}\rfloor$ connected components with equality only if the hypersurface has no real singular points.
\begin{cor}\label{cor:connectedness}
Let $V\subset \pp(\Sn)$ be a real $3$-dimensional linear subspace containing a strictly spectrahedral point. Then the real locus of $V\cap W_{n-1}$ contains at most $\lfloor \frac{n+1}{2}\rfloor$ connected components. Moreover, this bound is attained for some $V\subset \pp(\Sn)$ if $n=-1, 0, 1$ mod $8$. 
\end{cor}

For small $n$ we describe how many connected components the real locus of $V\cap W_{n-1}$ can have.
\begin{cor}
Let $V\subset \pp(\Sn)$ be as in Corollary \ref{cor:connectedness}. Then for $n<4$ the real locus of $V\cap W_{n-1}$ is connected, while for $n=4,5$ it is either connected or has two connected components.
\end{cor}
\begin{proof}
The cases of $n=2,3,4$ follow from the classification in \cite{quartic}. For $n=5$ Corollary \ref{cor:connectedness} gives an upper bound of three components. However, the zero set of a hyperbolic polynomial of degree $5$ can have three connected components only if its real part is smooth and Lemma \ref{lemma:rhopositive} shows that this is never the case for $V \subset \mathbb{P}(\mathbb{S}^5)$. It follows from Theorem \ref{thm:class} that there exist subspaces in $\pp(\mathbb{S}^5)$ of types $(\rho,\sigma)=(2,2)$ and $(\rho,\sigma)=(4,2)$. In the first case $(V\cap W_{4})(\R)$ has two connected components, while in the second case it is connected. See Figure \ref{fig:symms14} that illustrates the two possible cases.
\end{proof}

\begin{rem}
 Let $L\subset\pp(\SnR)$ be a one-dimensional real linear space which contains a strictly spectrahedral point. Then $L$ will contain $n$ points of rank $n-1$ (counted with multiplicity). Since for real symmetric matrices geometric multiplicities equal algebraic multiplicities, any intersection point of $L$ with $W_{n-1}$ of multiplicity $>1$ will already be a point in $W_{n-2}$.
 
 Naively, one might expect that the same happens for real linear subspaces $V\subset\pp(\SnR)$ of dimension $3$ containing a striclty spectrahedral point. Namely, that any intersection point of such $V$ with $W_{n-2}$ of multiplicity $>1$ has rank at most $n-3$. But this is not the case. For example let $V$ be the linear subspace spanned by the four matrices:
 \[\left(\begin{smallmatrix}2 & 0 & 0 & 0 & 0\\ 0 & 2 & 0 & 0 & 0\\ 0 & 0 & 1 & 0 & 0\\ 0 & 0 & 0 & 0 & 0\\ 0 & 0 & 0 & 0 & 0\end{smallmatrix}\right),
\left(\begin{smallmatrix}0 & 0 & 0 & 0 & 0\\ 0 & 0 & 0 & 0 & 0\\ 0 & 0 & 2 & 1 & 0\\ 0 & 0 & 1 & 2 & 0\\ 0 & 0 & 0 & 0 & 2\end{smallmatrix}\right),
\left(\begin{smallmatrix}1 & 0 & -1 & 1 & 0\\ 0 & 1 & 0 & 0 & 1\\ -1 & 0 & 2 & -1 & 0\\ 1 & 0 & -1 & 1 & 0\\ 0 & 1 & 0 & 0 & 1\end{smallmatrix}\right),
\left(\begin{smallmatrix}1 & 0 & 1 & 0 & 1\\ 0 & 1 & 0 & -1 & 0\\ 1 & 0 & 3 & 0 & 1\\ 0 & -1 & 0 & 1 & 0\\ 1 & 0 & 1 & 0 & 1\end{smallmatrix}\right).\]One checks that $V$ intersects $W_3$ in $5$ different points, all of them of multiplicity $4$, spectrahedral, and of rank exactly $3$. The associated symmetroid is the union of a cubic and a quadratic hypersurface which intersect in a sextic curve. This sextic curve is the union of two (nonreal) complex conjugate rational normal curves (of degree $3$). The $5$ points of rank $3$ are exactly the points where these two rational normal curves intersect.
\end{rem}

Let  $V\subset \pp(\Sn)$ be a generic real subspace of dimension $3$ containing a strictly spectrahedral point.
Every spectrahedral symmetroid is of the form $V\cap W_{n-1}$.
Recall that by $\rho$ and $\sigma$ we denote the number of real and spectrahedral nodes on the symmetroid $V\cap W_{n-1}$. The pair $(\rho,\sigma)$ is called the \emph{combinatorial type} of the symmetroid $V\cap W_{n-1}$ or, simply, of $V$.
Summarizing some of the above results we give necessary conditions on possible combinatorial types.

\begin{cor}\label{cor:conditions}
  The combinatorial type $(\rho,\sigma)$ of a transversal spectrahedral symmetroid of degree $n$ satisfies the following conditions:
  \begin{enumerate}[label=\roman*)]
  \item $0\leq \sigma\leq \rho\leq {n+1\choose 3}$,
  \item $\sigma$ is even,
  \item $\rho$ and ${n+1 \choose 3}$ are of the same parity,
  \item $\rho$ can be zero if and only if $n=-1,0,1\ \text{mod}\ 8$.
  \end{enumerate}
\end{cor}

In \cite{DI2011} and \cite{quartic} it was shown that for $n=4$ these are the only restrictions. Theorem \ref{thm:class} asserts that also for the quintic ($n=5$) case the above conditions on combinatorial types are sufficient. 
In order to prove sufficiency we provide explicit examples of symmetroids realizing each combinatorial type. We first describe in Section \ref{section:symmetry}  how we found some of the types by making use of symmetries. Realization of the remaing types is discussed in detail in Section \ref{section:computations}.

\section{Quintic spectrahedra with symmetry}\label{section:symmetry}
Let $G$ be a finite group. Let $V$ and $U$ be a $4$ resp. $5$-dimensional represention of $G$. If we have a $G$-linear map $\varphi: V\to\Sym_2(U)$, then the preimage under $\varphi$ of the set of positive semidefinite symmetric bilinear forms on $U$ is a spectrahedron invariant under the action of $G$. The same holds true for the preimage of the corank $k$ locus for any $k=0,\ldots,5$.

\begin{lemma}
 If there is a strictly spectrahedral point, then the invariant part $V^G$ of $V$ contains a strictly spectrahedral point. In particular, $V^G$ is nontrivial.
\end{lemma}

\begin{proof}
 Let $v\in V$ be a strictly spectrahedral point, i.e. $\varphi(v)$ is positive definite. Since the action of $G$ on $\Sym_2(U)$ preserves positive definiteness, the symmetric bilinear form $$\varphi\left(\sum_{g\in G} g.v\right)=\sum_{g\in G} g.\varphi(v)$$ is positive definite as well. Now the claim follows from $\sum_{g\in G} g.v\in V^G$.
\end{proof}

This shows that $V$ must contain (at least) one copy of the trivial represention of $G$. The remaining  $3$-dimensional space $V'$ corresponds to the symmetries of a $3$-polytope. One can check that neither octahedral symmetry nor icosahedral symmetry give rise to transversal quintic symmetroids.

\subsection{Tetrahedral symmetry}
The symmetry group of the tetrahedron is the symmetric group $\fS_4$ and the corresponding representation $V'$ is the standard representation. Thus $V$ is just given by $\fS_4$ permuting the coordinates of $\R^4$. If we denote by $V_\lambda$ the irreducible representation of $\fS_4$ corresponding to the partition $\lambda$, then $\R^4=V_4\oplus V_{3,1}$. In this case, the only possibility to obtain transversal symmetroids is when $U=V_{3,1}\oplus V_{2,2}$. After applying some suitable changes of bases on both sides, we can then assume that the images of the unit vectors in $\R^4$ under $\varphi$ are given by the four matrices
\[
 \left(\begin{smallmatrix}16+2t&24+t&24+t&-8&-8\\24+t&16+2t&24+t&8&0\\24+t&24+t&16+2t&0&8\\-8&8&0&16&8\\-8&0&8&8&16\end{smallmatrix}\right),
\left(\begin{smallmatrix}16+2t&-8+t&-8+t&8&8\\-8+t&-16+2t&-8+t&8&16\\-8+t&-8+t&-16+2t&16&8\\8&8&16&16&8\\8&16&8&8&16\end{smallmatrix}\right),\]\[
\left(\begin{smallmatrix}-16+2t&-8+t&-8+t&-8&8\\-8+t&16+2t&-8+t&-8&0\\-8+t&-8+t&-16+2t&-16&-8\\-8&-8&-16&16&8\\8&0&-8&8&16\end{smallmatrix}\right),
\left(\begin{smallmatrix}-16+2t&-8+t&-8+t&8&-8\\-8+t&-16+2t&-8+t&-8&-16\\-8+t&-8+t&16+2t&0&-8\\8&-8&0&16&8\\-8&-16&-8&8&16\end{smallmatrix}\right)
\] for some real number $t\in\R$. For generic values of $t$ there are $20$ (complex) points in the rank $3$ locus. In this case, the $20$ points split up into three $\fS_4$-orbits, one of length $12$ and two of length $4$, and we have the following cases.
If $48<t$, then the two $4$-orbits consist of semidefinite real points. The $12$-orbit consists of indefinite real points (Figure \ref{fig:symmtet}, yellow). 
If $12<t<48$, then one $4$-orbit consists of semidefinite real points and the other $4$-orbits of indefinite real points. The $12$-orbit consists of semidefinite real points (Figure \ref{fig:symmtet}, green).  
If $0<t<12$, then one $4$-orbit consists of semidefinite real points and the other $4$-orbit of indefinite real points. The $12$-orbit consists of nonreal points  (Figure \ref{fig:symmtet}, red).
If $-2<t<0$, then the two $4$-orbits consist of indefinite real points. The $12$-orbit consists of nonreal points.  The spectrahedron is empty (Figure \ref{fig:symmtet}, blue).
If $t<-2$, then none of the $20$ points are real. The spectrahedron is empty (Figure \ref{fig:symmtet}, purple). 
\begin{figure}[h]
\centering
 \includegraphics[width=3.3cm]{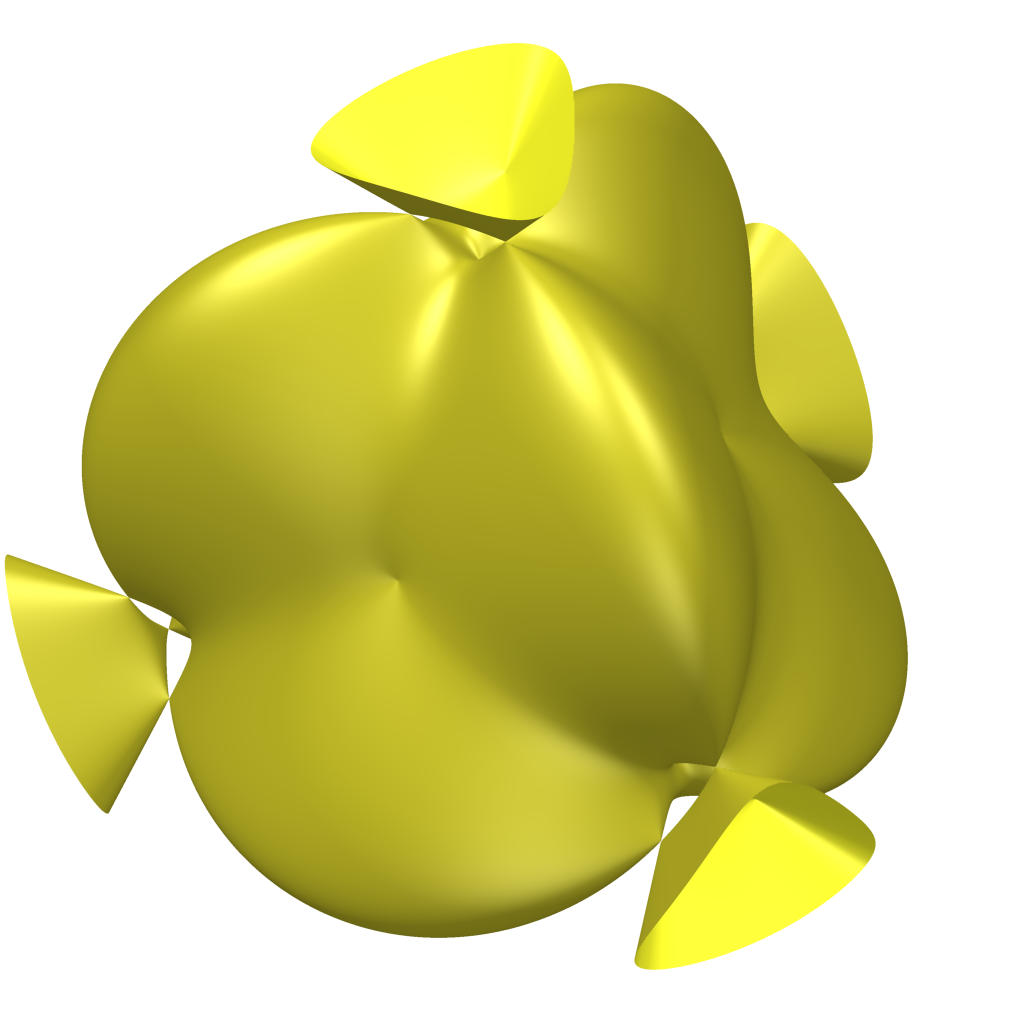} 
 \includegraphics[width=3.3cm]{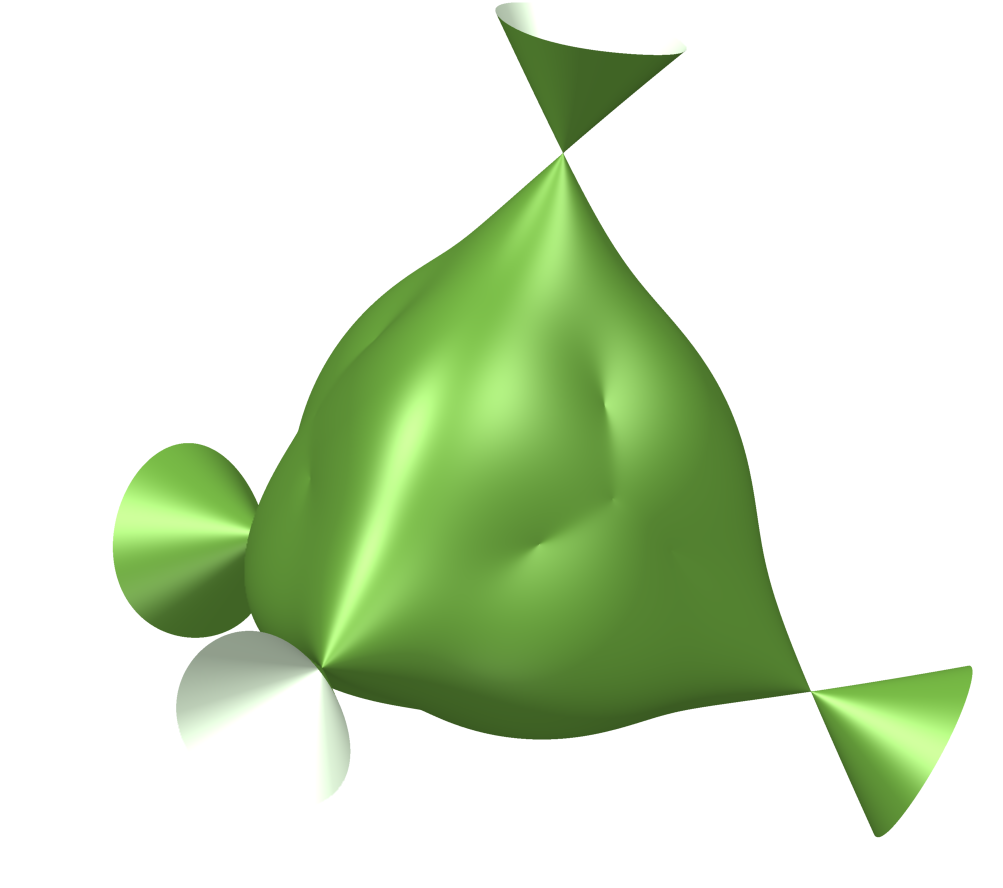}
 \includegraphics[width=3.3cm]{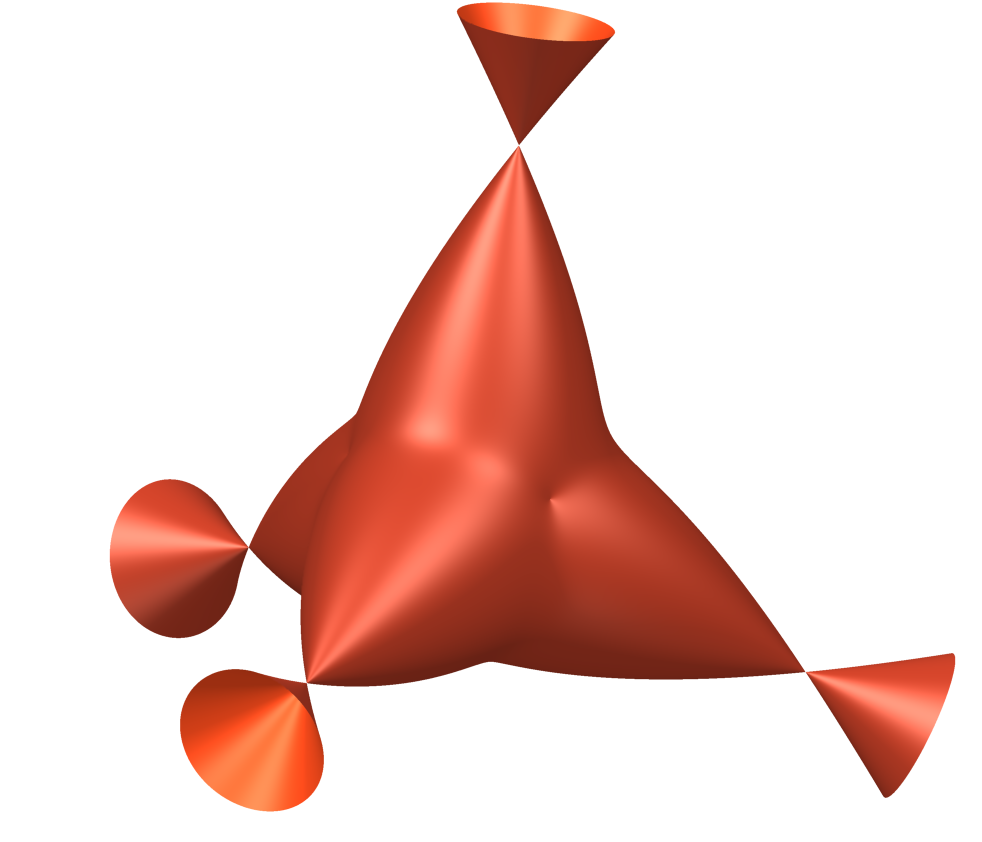} \\
 \includegraphics[width=3.3cm]{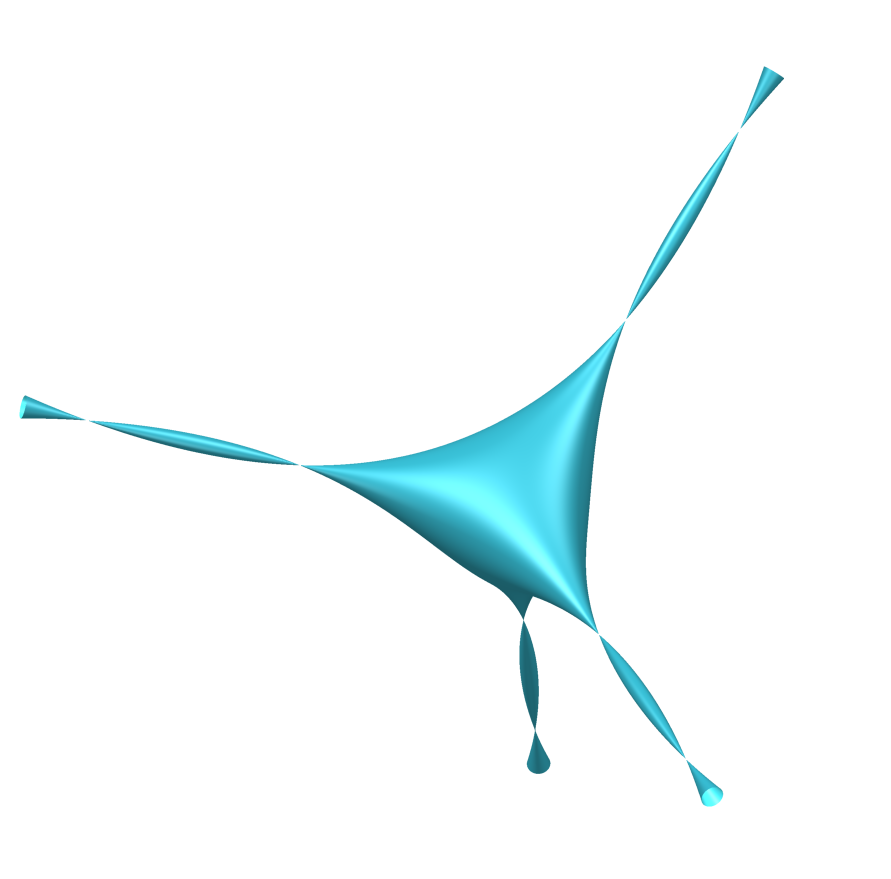}
 \includegraphics[width=3.3cm]{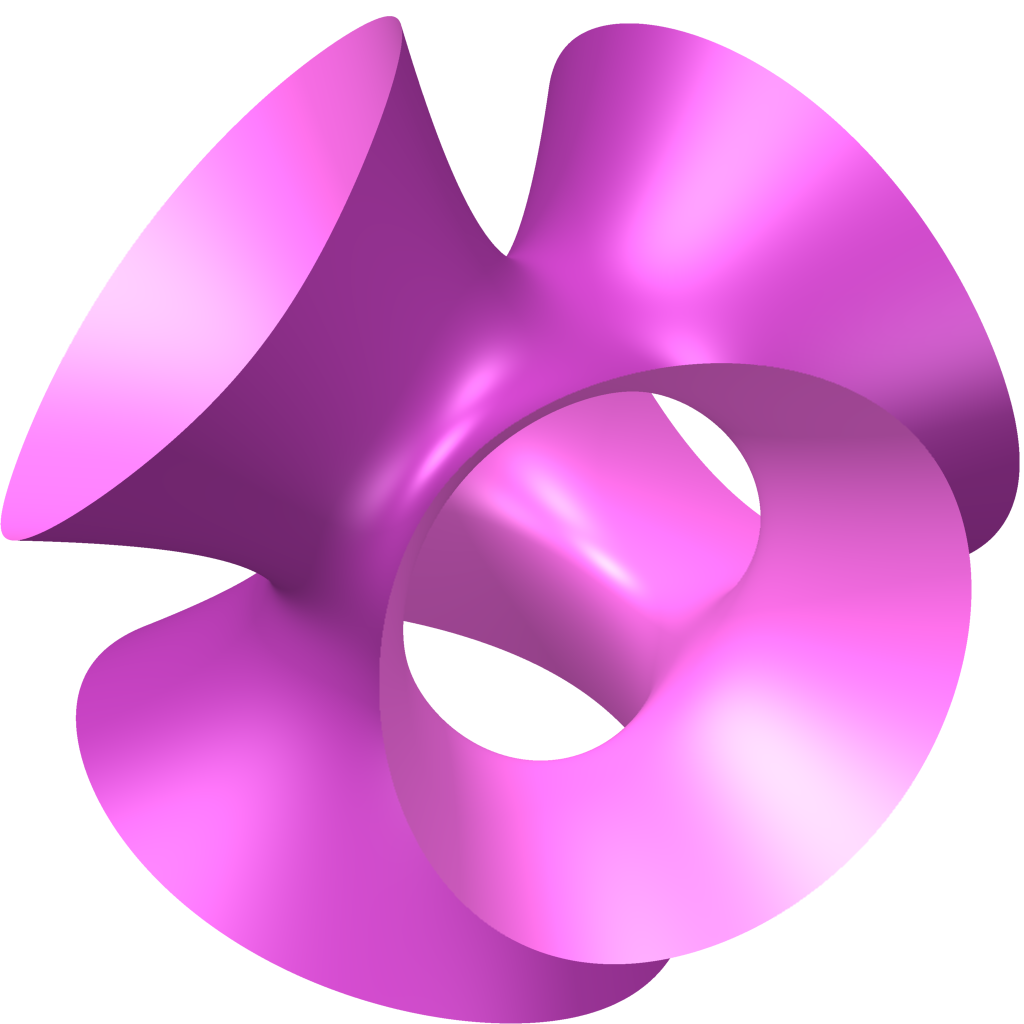}
\caption{Quintic symmetroids with tetrahedral symmetry.}
\label{fig:symmtet}
\end{figure}
\subsection{3-prismatic symmetry}
The symmetry group of the triangular prism is $G=\fS_3\times \fS_2$ where $\fS_3$ is the symmetry group of the triangle and $\fS_2$ corresponds to a reflection that interchanges the two triangles of the prism. Again one calculates that (after applying some suitable changes of bases on both sides) the images of the unit vectors in $\R^4$ under $\varphi$ are given by the four matrices
\[\left(\begin{smallmatrix}1& a& a& 0& 0\\ a& b+10& b+5& 0& 0\\ a& b+5& b+10& 0& 0\\ 0& 0& 0& 3& 2\\ 0& 0& 0& 2& 3\end{smallmatrix}\right),
\left(\begin{smallmatrix}1& -a& 0& 0& 0\\ -a& b+10& 5& 0& 0\\ 0& 5& 10& 0& 0\\ 0& 0& 0& 3& 1\\ 0& 0& 0& 1& 2\end{smallmatrix}\right),
\left(\begin{smallmatrix}1& 0& -a& 0& 0\\ 0& 10& 5& 0& 0\\ -a& 5& b+10& 0& 0\\ 0& 0& 0& 2& 1\\ 0& 0& 0& 1& 3\end{smallmatrix}\right),
\left(\begin{smallmatrix}0& 0& 0& 0& 0\\ 0& 0& 0& 2& 1\\ 0& 0& 0& 1& 2\\ 0& 2& 1& 0& 0\\ 0& 1& 2& 0& 0\end{smallmatrix}\right).\]
For generic values of $a,b$ there are $20$ (complex) points in the rank $3$ locus. In this case, the $20$ points split up into five $G$-orbits, one of length $2$, three of length $4$ and one of length $6$. The following table lists all different types we can get in this way together with one value for $(a,b)$ realizing this type:

\begin{figure}[!htpb]
\begin{minipage}{0.55\textwidth}
{\footnotesize
\begin{tabular}{|c|c|c|c|c|}
 \hline Type & Spectrahedron & a & b & Region \\
  &  nonempty? & & & \\
\hline (20,14) & Yes & $\frac{1}{2}$ & $-\frac{3}{2}$ & Brown\\
\hline (20,8) & Yes & 2 & 6 & Blue \\
\hline (18,6) & Yes & $\frac{1}{2}$ & 4 & Gray \\
 \hline(14,14) & Yes & $\frac{1}{2}$ & -2 & Orange \\
 \hline(14,8) & Yes & 1 & 1&Yellow \\
 \hline(14,2) & Yes & 1 & 6&Red \\
 \hline(10,4) & Yes & 1 & -2&Cyan\\
 \hline(8,2) & Yes & $\frac{1}{2}$ & 1&Green\\
 \hline(2,2) & Yes & $\frac{1}{3}$ & -2&Purple \\
 \hline(6,0) & No & 2 & -16 & n.a.\\
 \hline(0,0) & No & 1 & -16&n.a.\\ \hline
\end{tabular}
}
\end{minipage}
\begin{minipage}{0.45\textwidth}
\centering
  \includegraphics[width=4cm]{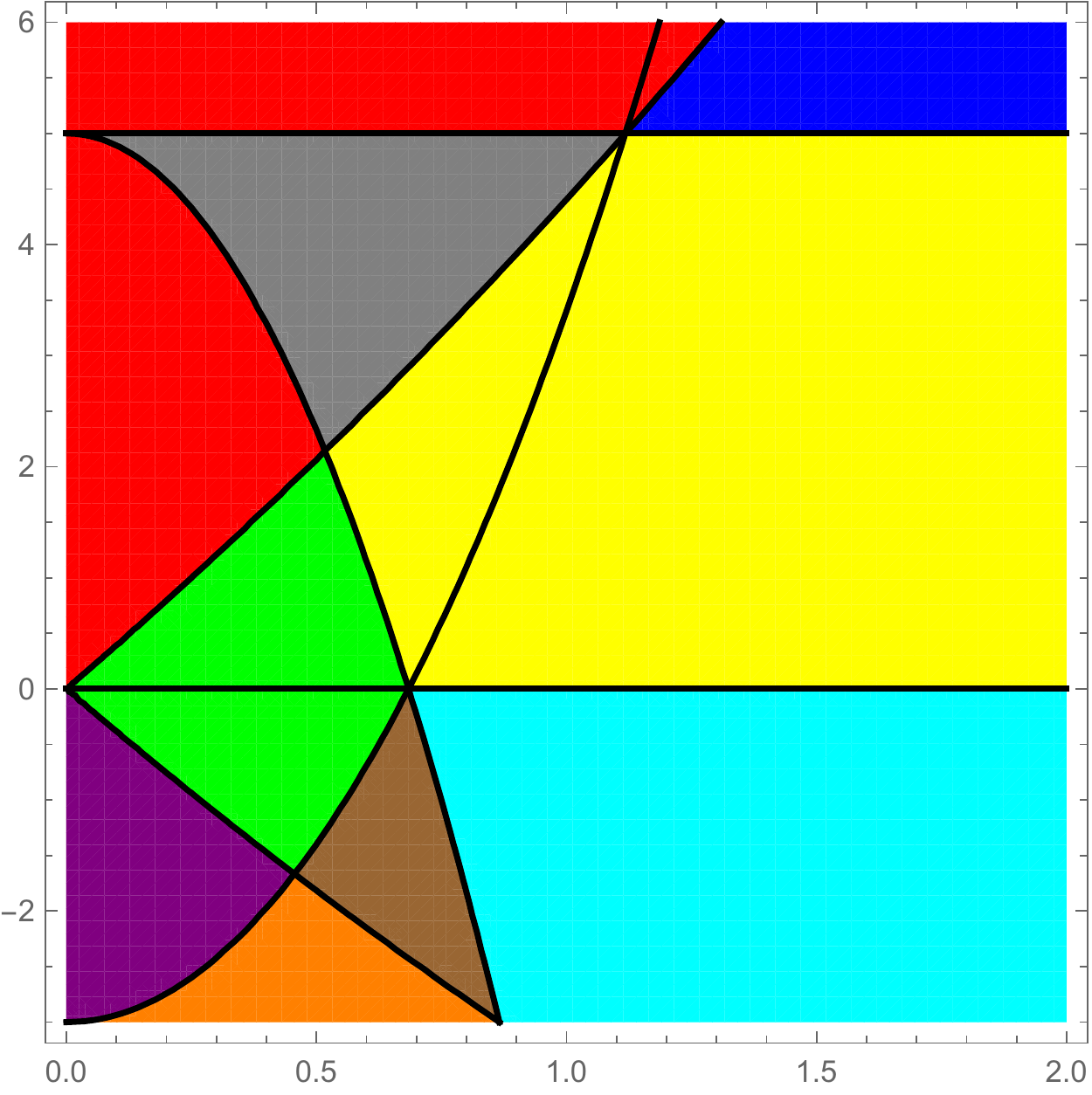}
  
  \caption{{\footnotesize The discriminant (black) divides the $(a,b)$-plane into different regions. The color of each region indicates the type of the symmetroids that it contains.}}
  \end{minipage}
 
\label{fig:symmetricpara}
\end{figure}

\begin{figure}[h!]
\centering \includegraphics[width=5cm]{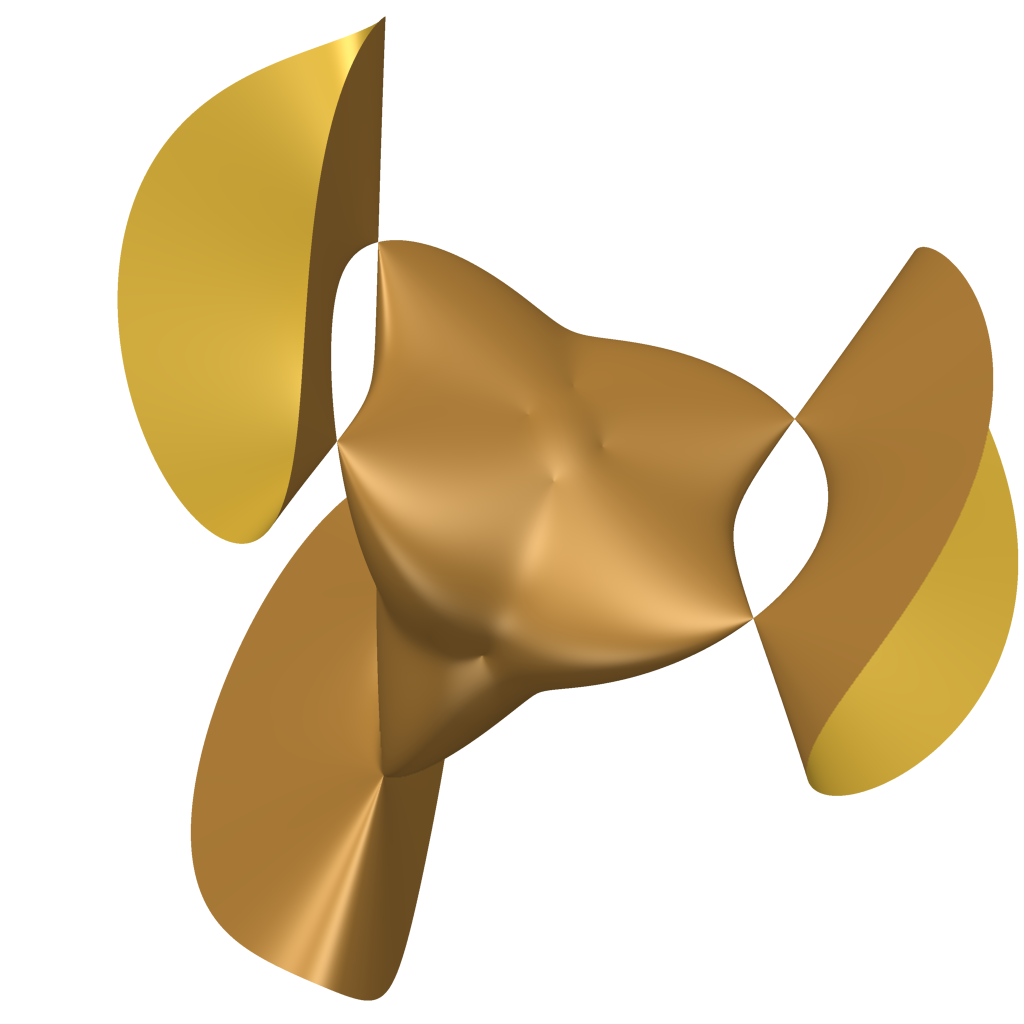} \quad
 \includegraphics[width=5cm]{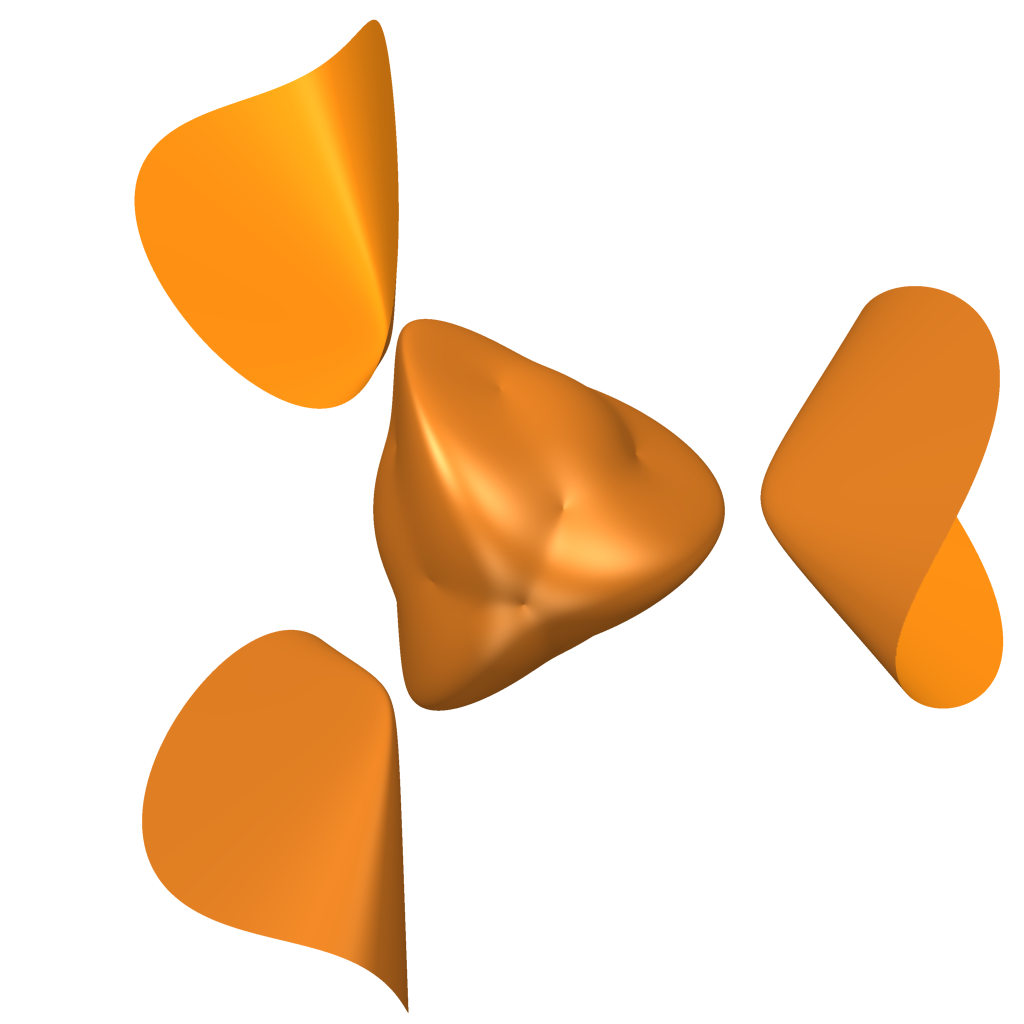} 
\caption{Two quintic symmetroids with $14$ nodes on the spectrahedron. The one on the left has $6$ additional real nodes. The real part of the one on the right is disconnected.}
\label{fig:symms14}
\end{figure}

\section{Computations}
\label{section:computations}
We show that the conditions in Theorem \ref{thm:class} are sufficient by exhibiting explicit examples of each possible combinatorial type. We explain how we used numerical algebraic geometry to find these examples and certify their correctness. First we describe equations which cut out the nodes of quintic symmetroids.
\subsection{Equations}

We consider full-dimensional spectrahedra so without loss of generality, we assume that our linear space of symmetric matrices $V_A\subset \mathbb{S}^5$ is of the form 
$$A(t,\bar{x}):=t \mathrm{Id} + x_1 A_1+x_2A_2+x_3A_3.$$
These linear spaces are parametrized by triples $A=(A_1,A_2,A_3) \in (\mathbb{S}^5)^3$ such that $\mathrm{Id},A_1,A_2,A_3$ are linearly independent. The symmetroid $V_A \cap W_4$ is the set of zeros of $D_A(t,\bar{x}):= \det(A(t,\bar{x}))$ and generically the $20$ nodes $V_A \cap W_3$ are the solutions to the polynomial system 
\begin{equation}
\label{eq:node_equations}
\frac{\partial}{\partial t} D_A(t,\bar{x}) = \frac{\partial}{\partial x_1} D_A(t,\bar{x}) = \frac{\partial}{\partial x_2} D_A(t,\bar{x}) = \frac{\partial}{\partial x_3} D_A(t,\bar{x}) =0 . 
\end{equation}
\begin{rem}
\label{rem:atleasttwenty}
Since the solutions to this system are points in $\mathbb{P}^3$, there are always at least $20$ complex projective solutions to \eqref{eq:node_equations} counted with multiplicity. This number may be larger (or infinite) at nongeneric values of $A$, but never smaller.
\end{rem}
Although the solutions to \eqref{eq:node_equations} are projective points, in practice we compute them in a generic real affine chart represented by an equation of the form
$\ell(t,\bar{x})=c_0t+c_1x_1+c_2x_2+c_3x_3+c_4$, where $c_i \in \mathbb{R}$.
Let $F_A(t,\bar{x})$ denote the polynomial system \eqref{eq:node_equations} together with $\ell(t,\bar{x})=0$.
Each solution $s$ to $F_A(t,\bar{x})$ corresponds to a matrix $A(s)$. When $A \in (\mathbb{S}^5(\mathbb{\R}))^3$, the following conditions concerning reality coincide:
$$s \in \mathbb{R}^4 \iff A(s) \in \mathbb{S}^5(\mathbb{R}) \iff [A(s)] \cap \mathbb{P}(\mathbb{S}^5(\mathbb{R})) \neq \emptyset.$$
One checks whether $A(s)$ is semidefinite either by checking that its eigenvalues have the same sign, or by assessing the signs of its principal minors.  

\subsection{Numerical algebraic geometry}
We use the \texttt{julia} package \texttt{HomotopyContinuation.jl} \cite{jl} to solve instances of $F_{A}(t,\bar{x})$ for various parameters $A$. The equations $F_{A}(t,\bar{x})$ form a parametrized polynomial system which can be solved using a \emph{parameter homotopy}. This process involves numerically solving $F_{A^{0}}(t,\bar{x})$ for random \emph{complex} symmetric matrices $A^{0}$. Afterwards, the method of \emph{homotopy continuation} is used to solve any instance of $F_A(t,\bar{x})$ with $A \in (\mathbb{S}^5(\mathbb{R}))^3$ quickly using a \emph{parameter homotopy} from the start system $F_{A^0}(t,\bar{x})$ to the target system $F_A(t,\bar{x})$. The following snippet of code solves $20000$ instances of $F_A(t,\bar{x})$ in about $30$ minutes.

\begin{figure}[h]
	\centering
	\vspace{-0.06in}	
	\begin{tiny}
		\begin{BVerbatim}	
using HomotopyContinuation, LinearAlgebra

@var x[1:4], a[1:3,1:5,1:5]
A=vcat([LinearAlgebra.I],[Symmetric(a[i,:,:]) for i in 1:3])
Ax=sum([x[i]*A[i] for i in 1:4])
affine_chart=sum(x.*randn(Float64,4))+randn()
eqs=vcat([differentiate(det(Ax),x[i]) for i in 1:4],affine_chart)
quintic_system=System(eqs;parameters=unique(vcat(vec(A[2]),vec(A[3]),vec(A[4]))))
init_params=randn(ComplexF64,45)
init_sols=solve(quintic_system,target_parameters=init_params)
samples=solve(quintic_system,
              solutions(init_sols),
              start_parameters=init_params,
              target_parameters=[randn(Float64,45) for i in 1:20000]);
Solving for 20000 parameters... 100%|||||||||||||||||| Time: 0:31:51
  # parameters solved:  20000
  # paths tracked:      400000
		\end{BVerbatim}
	\end{tiny}
\end{figure}

\subsection{A hill-climbing algorithm}
Finding the maximal number of real solutions in a parametrized polynomial system is a popular problem in applications. One very successful method for doing this is a hill-climbing algorithm developed by Dietmaier \cite{dietmaier}. We adapt his idea to our setting, where we care about more than the number of real solutions to $F_A(t,\bar{x})$: we want to find all combinatorial types.

To best explain our method, we first describe what a hill-climbing algorithm is.
The input of a hill-climbing algorithm is a parameter space $\mathcal P$ equipped with a metric, and an objective function $\varphi: \mathcal P \to R$ where $R$ is some ordered set. The goal of the algorithm is to maximize $\varphi$ over $\mathcal P$ and the steps are simple. Begin at some $P_* \in \mathcal P$ and compute $\varphi_*=\varphi(P_*)$. Compute the values of $\varphi$ evaluated at many parameters near $P_*$ and update $P_*$ to be the one which evaluates to the largest value when this is larger than $\varphi_*$, then repeat. In this fashion, the algorithm greedily travels through $\mathcal P$ increasing the value of $\varphi$.

In \cite{dietmaier}, $\mathcal P$ is the parameter-space of a real parametrized polynomial system $G_{\mathcal P}=0$. The set $R$ is $\mathbb{N} \times \R$ equipped with the lexicographic ordering. The objective function takes a parameter $P$ to $(\rho,-m)$ where $\rho$ is the number of real solutions to $G_P=0$, and $m$ is the minimum norm of the complex part of a nonreal solution of $G_P=0$. Consequently, this algorithm tries to greedily push complex conjugate solutions together to make them real, while always prefering parameters with more real solutions.

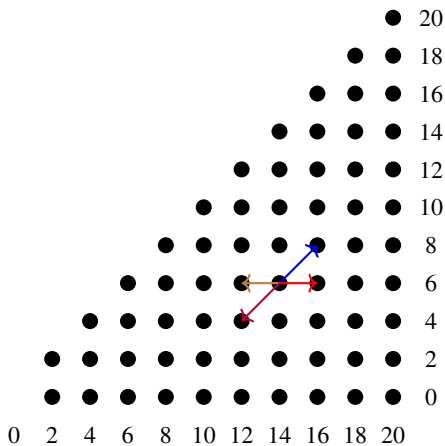
\begin{figure}[!htpb]
\begin{center}
  \begin{tikzpicture}[scale=0.5]
    \foreach \i in {0,2,4,6,8,10,12,14,16,18,20}
      \path[black] (\i/2,-1) node{\footnotesize{\i}} (11,\i/2) node{\footnotesize{\i}};
    \foreach \i in {0,2,4,6,8,10,12,14,16,18,20}
      \foreach \j in {0,2,4,6,8,10,12,14,16,18,20}{
      \if \i=0
      	\fill[black] (4,0) circle(6pt);
      \else
      	\ifnum \i> \j
        	\fill[black] (\i/2,\j/2) circle(6pt);
        \fi

      \fi
      };
      \foreach \i in {2,4,6,8,10,12,14,16,18,20}{
      	  \fill[black] (\i/2,\i/2) circle(6pt);
      	  };
      \draw[blue,thick,->] (7,3) -- (8,4);
      \draw[red,thick,->] (7,3) -- (8,3);
      \draw[brown,thick,->] (7,3) -- (6,3);
      \draw[purple,thick,->] (7,3) -- (6,2);
  \end{tikzpicture} 
  \caption{All combinatorial types of quintic spectrahedra.}
  \label{fig:alltypes}
\end{center}
  \end{figure}

To find each combinatorial type in Figure \ref{fig:alltypes}, we take a similar approach.  After solving $F_A(t,\bar{x})$ for some parameter $A$, we use four different hill-climbing algorithms, each of which attempts to move the combinatorial type of $V_A$ in one of the directions displayed in Figure \ref{fig:alltypes}.

We introduce some notation. For any two combinatorial types $\tau=(\rho,\sigma)$ and $\tau'=(\rho',\sigma')$, let $d(\tau,\tau')$ be their lattice distance in the lattice generated by the vectors in Figure \ref{fig:alltypes}. Given a nonreal matrix $B=B_\mathbb{R}+\sqrt{-1}B_{\mathbb{C}}$, let $\mu(B)=||B_{\mathbb{C}}||_2$. Let $\eta(B)=1$ if the largest three eigenvalues of $B_{\mathbb{R}}$ have the same sign and $0$ otherwise. For $A \in (\mathbb{S}^5(\mathbb{R}))^3$ let $S_{\R_+}(A)$ and $S_{\R_-}(A)$ denote the real semidefinite and real indefnite solutions to $F_A(t,\bar{x})$. Similarly, let $S_{\C_+}(A)$ denote the collection of nonreal solutions $B$ with $\eta(B)=1$ and $S_{\C_-}(A)$ the nonreal solutions with $\eta(B)=0$. For every $A \in (\mathbb{S}^5(\mathbb{R}))^3$, we have the functions
\begin{align*}
\mu_+(A) &= \min_{B \in S_{\C_+}(A)}(\mu(B)) \hspace{0.5 in }\delta_+(A) = \min_{B_1,B_2 \in S_{\R_+}(A)}(||B_1-B_2||_2)\\
\mu_-(A) &= \min_{B \in S_{\C_-}(A)}(\mu(B))\hspace{0.5 in }\delta_-(A) = \min_{B_1,B_2 \in S_{\R_-}(A)}(||B_1-B_2||_2)
\end{align*}
which formally evaluate to $+\infty$ when the minimum is taken over the empty set.

With these definitions, the following table  describes the four hill-climbing algorithms used to move from some type $\tau$ to another type $\tau' = \tau+$Direction. 

\begin{center}
{
\begin{tabular}{|c|c|c|c|}
\hline 
Direction & $R$ & $\varphi$ & Ordering\\
\hline \hline
$(+2,+2)$ & $\mathbb{N} \times \mathbb{R} \cup \{\infty \}$ & $(d(\tau',(\rho,\sigma)),-\mu_+(A))$& Lex\\
\hline
$(+2,0)$ &  $\mathbb{N} \times \mathbb{R} \cup \{\infty \}$  & $(d(\tau',(\rho,\sigma)),-\mu_-(A))$ & Lex \\
\hline
$(-2,-2)$ &  $\mathbb{N} \times \mathbb{R} \cup \{\infty \}$  & $(d(\tau',(\rho,\sigma)),-\delta_+(A))$ & Lex\\
\hline
$(-2,0)$ &  $\mathbb{N} \times \mathbb{R} \cup \{\infty \}$ & $(d(\tau',(\rho,\sigma)),-\delta_-(A))$ & Lex \\
\hline
\end{tabular}
}
\end{center}

Rows correspond to arrows in Figure \ref{fig:alltypes} and describe algorithms which attempt to push two solutions (in $S_{\C_+}$,$S_{\C_-}$, $S_{\R_+}$, or $S_{\R_-}$ respectively) together.

Since hill-climbing methods get stuck at local minima, it was often necessary to restart the algorithm with new initial values many times until one became successful. For some of the most difficult cases like those near $(20,20)$ and $(20,0)$, monitoring of the algorithm was required due to ill-conditioning.

\subsection{Certification}

Solving $F_A(t,\bar{x})$ numerically only provides approximate solutions and so \emph{a posteriori} certification methods are necessary in order to verify how many of these solutions are real, how many correspond to semidefinite matrices, and ultimately whether the solutions are even correct. 

Certifying the correctness and reality of a set of solutions to a polynomial system is usually simple and can be performed via the command \texttt{certify} \cite{BRT} in \texttt{HomotopyContinuation.jl}. The input of this command is a collection $S$ of approximate solutions to a polynomial system $F$, and when successful, the output is a collection of $|S|$ bounding boxes, each containing a unique solution to $F$. This method, however, cannot be immediately applied to our situation because the system $F_A(t,\bar{x})$ is not \emph{square}: there are more equations than variables. Nor does this method verify which solutions correspond to semidefinite matrices.

To remedy these issues, we solve a different polynomial system  
$G_A(t,\bar{x},d,M)$:
\begin{equation}
\begin{aligned}  
\frac{\partial}{\partial x_1} D_A(t,\bar{x}) &=0 \hspace{0.33 in} \frac{\partial}{\partial x_2} D_A(t,\bar{x}) =0 \hspace{0.175 in} \frac{\partial}{\partial x_3} D_A(t,\bar{x}) 
=0\\
\ell(t,\bar{x}) &= 0, \hspace{0.32 in} \det(A(t,\bar{x})) = d,\hspace{0.2 in}  (A(t,\bar{x}))_I = M_{I}.
\end{aligned}
\end{equation}
The system $G_A(t,\bar{x},d,M)$ consists of three out of four of the equations of \eqref{eq:node_equations}, the affine chart $\ell(t,\bar{x})$, an equation tracking the value of the determinant of $A(t,\bar{x})$, and equations tracking each of the values of the $5+10+10=25$ principal minors $M_I$ indexed by subsets $I$ of $\{1,\ldots,5\}$ of sizes $1$, $2$, and $3$. In total, $G_A(t,\bar{x},d,M)$ is a polynomial system consisting of $30$ equations in $30$ unknowns. Since it is square, we can certify its solutions.

The total number of solutions to the first three equations is $64$. This is their B\'ezout bound. The other equations and variables form a graph of a function over this set. When distinct, at least $20$ of these $64$ solutions project to solutions of $F_A(t,\bar{x})$ by Remark \ref{rem:atleasttwenty}. Those which do are the points whose $d$-coordinate equals zero. Therefore, showing that $44$ of the solutions do not have $0$ in their certified bounding interval in the $d$-coordinate proves that the other $20$ solutions have determinant equal to zero. Similarly, one certifies the signs of the minors whenever $0$ is not contained in the corresponding bounding intervals. These facts suggest an algorithm for certifying combinatorial.
\begin{algorithm}
\label{alg:certification}
{\color{white}{.}}\\
{\bf Input}: Three real symmetric matrices $A=(A_1,A_2,A_3).$\\
{\bf Output}: The combinatorial type $(\rho,\sigma)$ of $V_A$, or \texttt{Unsuccessful}.
\begin{enumerate}
\itemsep-0.25em 
\item Solve $G_A(t,\bar{x},d,M)$ numerically to obtain $S$, consisting of $64$ solutions.
\item Run \texttt{certify} to obtain $64$ bounding boxes for these solutions.
\item Delete $44$ bounding boxes which do not contain $0$ in the $d$-coordinate. 
\item Count $\rho$ to be the number of remaining boxes which are real. 
\item For each real box, determine the signs of the $M$-coordinates.
\item Count $\sigma$ to be the number of boxes whose $M$-coordinates imply semidefiniteness of the solution.
\item If successful, return $(\rho,\sigma)$, otherwise, return \texttt{Unsuccessful}.
\end{enumerate}

\end{algorithm}
We successfully performed Algorithm \ref{alg:certification} on witnesses for each of the $65$  combinatorial types. Certificates and necessary software for reproducibility may be found at 
\begin{equation} \label{eq:ourwebsite}
\href{https://mathrepo.mis.mpg.de/QuinticSpectrahedra/index.html}{https://mathrepo.mis.mpg.de/}
\end{equation}
Additionally, \eqref{eq:ourwebsite} contains a gallery of images displaying examples of the $65$ combinatorial types. Figure \ref{fig:spec2020} shows two extremal cases: $(20,20)$ and $(20,0)$.

\begin{figure}[!htpb]
\centering
\includegraphics[width=6.2cm]{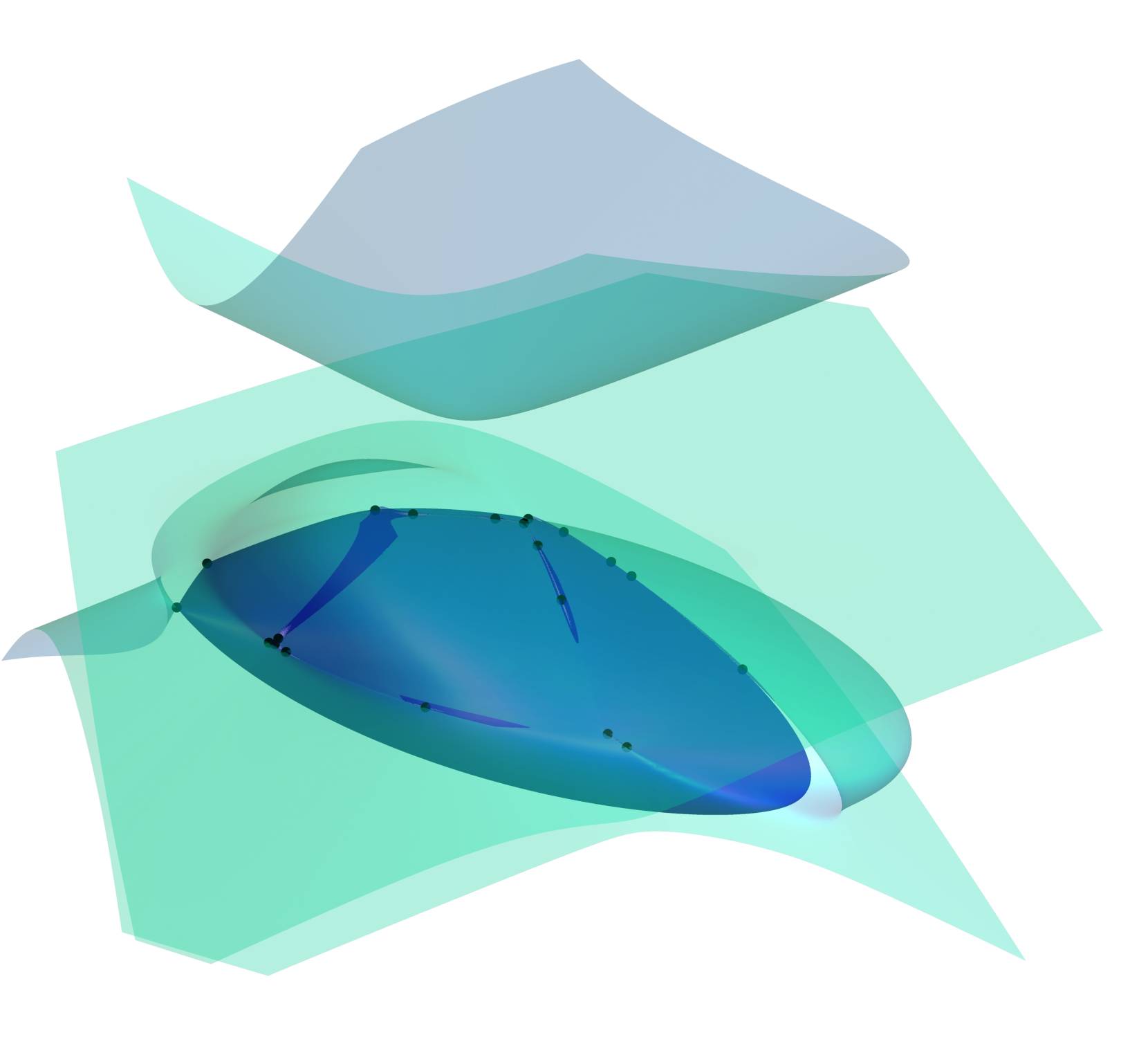}
\hspace{-0.5cm}
\includegraphics[width=6.2cm]{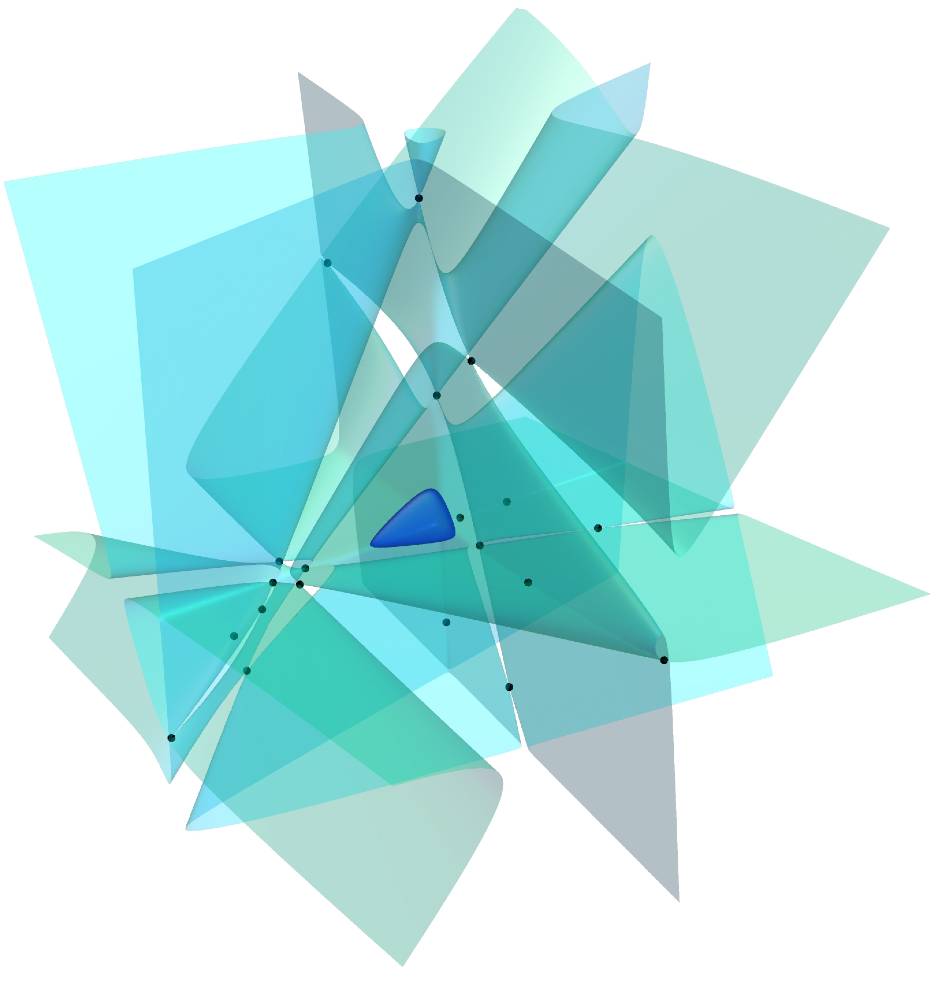}
\caption{Quintic spectrahedra of types $(20,20)$ and $(20,0)$.}
\label{fig:spec2020}
\end{figure}

\section*{Acknowledgements}
This work is a part of the collaboration project ``Linear Spaces of Symmetric Matrices'' at MPI MiS and worldwide. We would like to thank Orlando Marigliano, Mateusz Micha{\l}ek, Kristian Ranestad, Tim Seynnaeve, and Bernd Sturmfels for coordinating the project and inspiring the present work. We would also like thank Cynthia Vinzant for her help in visualizing spectrahedra, as well as Sascha Timme and Paul Breiding for their help with the certification.

\end{document}

%% file: operators.tex
\usepackage{mathrsfs}

\newcommand{\Sn}{\mathbb{S}^{n}}
\newcommand{\SnR}{\mathbb{S}^{n}(\mathbb{R})}

\newcommand{\Sym}{\operatorname{Sym}}

\newcommand{\codim}{\operatorname{codim}}

\newcommand{\tr}{\operatorname{tr}}



%% file: fonts.tex
\newcommand{\fS}{{\mathfrak S}}

\newcommand{\C}{{\mathbb C}}
\newcommand{\R}{{\mathbb R}}
\newcommand{\pp}{\mathbb{P}}

\newcommand{\Z}{{\mathbb Z}}

